\pdfoutput=1 

\documentclass[11pt]{amsart}
\usepackage{amsmath,amsthm,amssymb,amscd,eucal}        
\usepackage{graphicx}
\usepackage[all]{xy}
\usepackage[margin=1.3in]{geometry}

\begin{document}

\newcommand{\half}{{\textstyle{\frac{1}{2}}}}
\newcommand{\balpha}{{\boldsymbol{\alpha}}}
\newcommand{\B}{{\bf 1}}
\newcommand{\e}{{\bf e}}
\newcommand{\y}{{\bf r}}
\newcommand{\x}{{\bf x}} 
\newcommand{\tx}{{\bar{x}}}                   
\newcommand{\bn}{{\bf n}}

\newcommand{\bS}{{\bf S}}
\newcommand{\eg}{{\rm eg}}
\newcommand{\mor}{{\rm Mor}}
\newcommand{\w}{{\bf v}}
\newcommand{\tS}{{\widetilde{S}}}

%
%

\newcommand{\la}{\langle}
\newcommand{\ra}{\rangle}

\newcommand{\aff}  {{\rm Aff}_+}                 
\newcommand{\Aff}  {{\rm Aff}}                 

\newcommand{\Conf}{{\rm Config}}                 
\newcommand{\conf}{{\rm config}}                 
 
\newcommand{\PSL} {\Pj\Sl_2(\R)}                           
\newcommand{\PGL} {\Pj\Gl_2(\R)}                           
\newcommand{\PGLC} {\Pj\Gl_2(\C)}                          
\newcommand{\RP} {\R\Pj^1}                                 
\newcommand{\CP} {\C\Pj^1}                                 
\newcommand{\C} {{\mathbb C}}                              
\newcommand{\R} {{\mathbb R}}                              
\newcommand{\Z} {{\mathbb Z}}                              
\newcommand{\Pj} {{\mathbb P}}                             
\newcommand{\Sg} {{\mathbb S}}                             
\newcommand{\Gl} {{\rm Gl}}                                
\newcommand{\Sl} {{\rm Sl}}                                
\newcommand{\SO} {{\rm SO}}                                
\newcommand{\sM}{{\sf M}}
\newcommand{\wsM}{{\widetilde{\sf M}}}
\newcommand{\orb}{{\rm orb}}
\newcommand{\PGl}{{\rm PGl}} 
\newcommand{\SP}{{\rm SP}}
\newcommand{\Ot}{{\rm O}}
\newcommand{\D}{{\rm D}}
\newcommand{\reg}{{\rm reg}}

\newcommand{\cA}{{\mathcal A}}
\newcommand{\cB}{{\mathcal B}}
\newcommand{\cC}{{\mathcal C}}
\newcommand{\blambda}{{\boldsymbol{\lambda}}}
\newcommand{\bDelta}{{\boldsymbol{\Delta}}}
\newcommand{\bdelta}{{\boldsymbol{\delta}}}
\newcommand{\cQ}{{\mathcal Q}}
\newcommand{\dQ}{{\mathcal Q}^0}   
\newcommand{\qQ}{{\widehat{\cQ}}}
\newcommand{\tQ}{\widetilde{\cQ}}
\newcommand{\sQ}{{\sf Q}}
\newcommand{\tsQ}{{\widetilde{\sf Q}}} 

\newcommand{\cell}{{\rm Cell}}
\newcommand{\cells}{{\rm Cells}}
\newcommand{\iso}{{\rm Iso}}

\newcommand{\PV}  [1] {\ensuremath{\Pj V^{#1}}}
\newcommand{\SV}  [1] {\ensuremath{\Sg V^{#1}}}
\newcommand{\PVM} [1] {\ensuremath{\Pj V^{#1}_{\#}}}
\newcommand{\SVM} [1] {\ensuremath{\Sg V^{#1}_{\#}}}

\newcommand{\metmap}{\mathfrak{d}}

\newcommand{\dis} {{\mathcal D}}                                 
\newcommand{\diagonal}{\triangle}                                

\newcommand{\oM} [1] {\ensuremath{{\mathcal M}_{0,#1}(\R)}}             					
\newcommand{\M} [1] {\ensuremath{{\overline{\mathcal M}}{_{0, #1}(\R)}}}					
\newcommand{\oOM} [1] {\ensuremath{{\mathcal M}_{0,#1}^{{\rm or}}(\R)}}					
\newcommand{\OM} [1] {\ensuremath{{\overline{\mathcal M}}{_{0, #1}^{{\rm or}}(\R)}}}		
\newcommand{\KM} [1] {\ensuremath{{\overline{\mathcal M}}{_{0, #1}^{{\rm kap}}(\R)}}}		
\newcommand{\cM} [1] {\ensuremath{{\mathcal M}_{0, #1}}}                  					
\newcommand{\CM} [1] {\ensuremath{{\overline{\mathcal M}}_{0, #1}}}       					
\newcommand{\iDelta}{{\overset{o}\Delta}}

\newcommand{\bT}{{\overline{\mathcal T}}}
\newcommand{\T}[1]{{\mathcal T}_#1} 
\newcommand{\pT}[1]{{\bf T}_{#1}(\R)}             
\newcommand{\BHV}[1]{{\rm BHV}_{#1}}
\newcommand{\iBHV}[1]{{\rm BHV}_{#1}^+} 

\newcommand{\tsM} [1]{\ensuremath{{\widetilde{\sf M}}{_{0, #1}(\R)}}}
\newcommand{\gM} [1] {{\sf M}_{#1}}
\newcommand{\tgM} [1] {\widetilde{\sf M}_{#1}}

\newcommand{\hide}[1]{}

\newcommand{\suchthat} {\:\: | \:\:}
\newcommand{\ore} {\ \ {\it or} \ \ }
\newcommand{\oand} {\ \ {\it and} \ \ }

%
%

\theoremstyle{plain}
\newtheorem{thm}{Theorem}
\newtheorem{prop}[thm]{Proposition}
\newtheorem{cor}[thm]{Corollary}
\newtheorem{lem}[thm]{Lemma}
\newtheorem{conj}[thm]{Conjecture}
\newtheorem*{arnlem}{Arnold's Lemma}

\theoremstyle{definition}
\newtheorem*{defn}{Definition}
\newtheorem*{exmp}{Example}

\theoremstyle{remark}
\newtheorem*{rem}{Remark}
\newtheorem*{hnote}{Historical Note}
\newtheorem*{nota}{Notation}
\newtheorem*{ack}{Acknowledgments}
\numberwithin{equation}{section}


\title{Diagonalizing the Genome II: Toward Possible Applications}

\author{Satyan L.\ Devadoss}
\address{S.\ Devadoss: Williams College, Williamstown, MA 01267}
\email{satyan.devadoss@williams.edu}

\author{Jack Morava}
\address{J.\ Morava: Johns Hopkins University, Baltimore, MD 21218}
\email{jack@math.jhu.edu}

\begin{abstract}
In a previous paper, we showed that the orientable cover of the moduli space of real genus zero 
algebraic curves with marked points is a compact aspherical manifold tiled by associahedra, 
which resolves the singularities of the space of phylogenetic trees. In this draft of a sequel,
we construct a related (stacky)
resolution of a space of real quadratic forms, and suggest,
perhaps without much justification, that systems of oscillators parametrized by such objects may may provide
useful models in genomics. 
\end{abstract}



\maketitle

\baselineskip=17pt

%
%
\section{Introduction} 
\label{s:intro}

One of the vexing problems in theoretical biology is the relation between genotype (easily measurable by
DNA sequencing) and phenotype (less easily defined, or measured): this resembles old questions in quantum
mechanics about observables and hidden variables. It may be naive, but nevertheless worthwhile, to suggest
that there may be interesting connections between models of evolution based on resonances in systems of linked
oscillators, and the inverse problem of reconstructing evolutionary trees by dissimilarity matrix 
techniques.\footnote{For example, one might hope for a metric version of the topological stability theorem of \cite{lp}.}
In the language proposed here, this becomes a question about maps between moduli spaces of quadratic forms.

In previous work, we have constructed a manifold resolution of the space $\iBHV{n}$ of phylogenetic 
trees introduced for classification problems in big data, such as in genomics. Our resolution $\OM{n+1}$
adapts a construction from algebraic geometry of a moduli space for configurations of points on
the real line (up to projective equivalence), building on a duality between cubical and associahedral
tessellations of certain hyperbolic manifolds. We refer to \cite[\S 5]{dm} for the details of its 
construction.  

This follow-up paper suggests applications of Part I to the study of configurations of eigenvalues of 
real symmetric matrices.  In particular, we construct an \emph{orbihedral} \cite{ah} cover $\tsQ^*_n$
of a space $\cQ^*_n$ of such (suitably normalized) matrices, whose elements can be interpreted as
matrices with \emph{labelled} 
eigenvalues. In physics, eigenvalues of symmetric matrices often 
represent fundamental frequencies of mechanical systems, such as spectral lines.  
We propose here to think of these labelled eigenvalues
as marking systems of coupled oscillators which retain
some integrity over time, organized as loci on a circle;
such systems abound in genetics, eg.\ in mitochondria or
in bacteria such as \emph{E.\ coli}. Figure~\ref{f:circos} is a
representative example.

\begin{figure} [h]
\includegraphics[width=.7\textwidth]{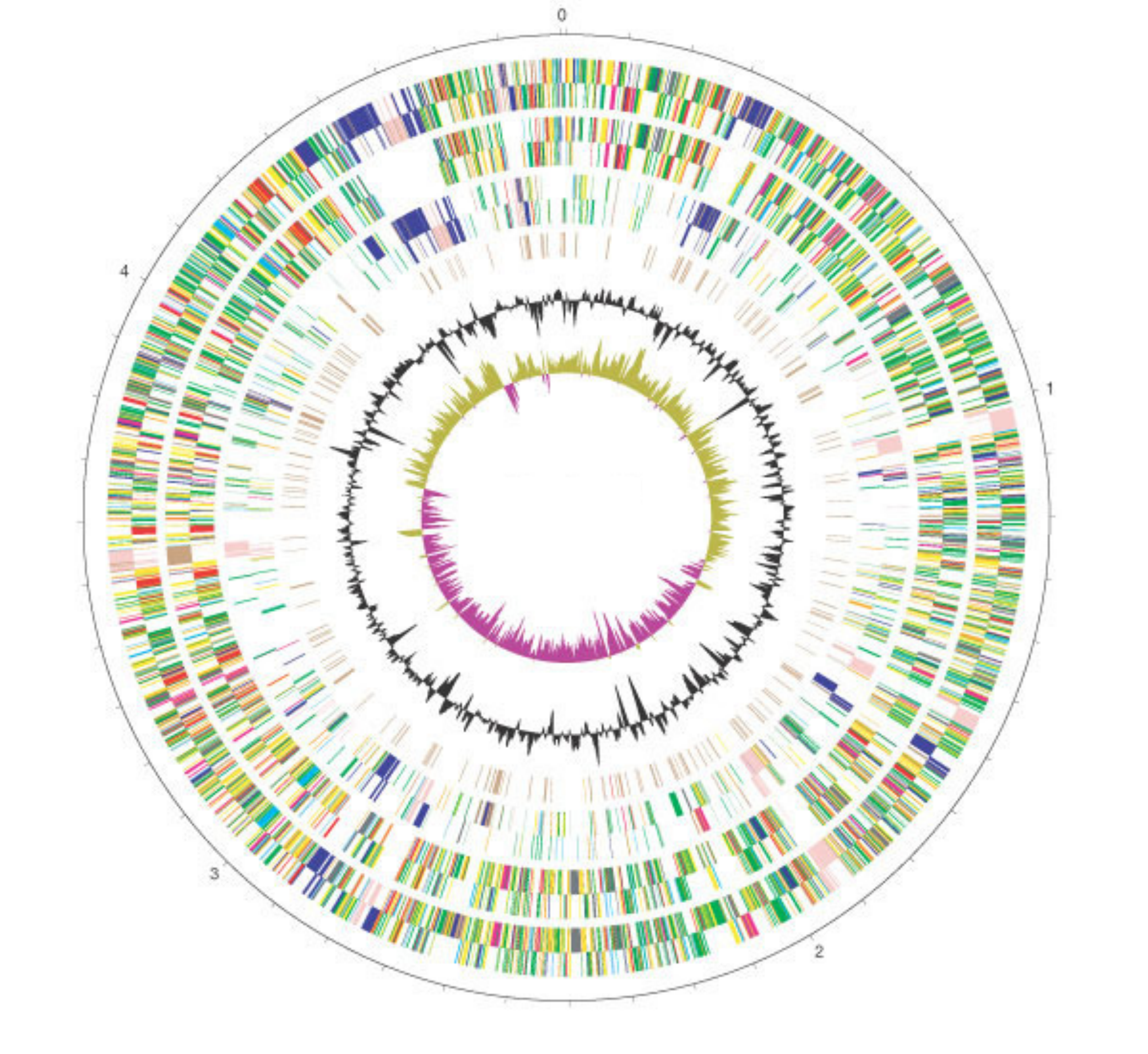}
\caption{Complete genome sequence of a multiple drug resistant \emph{Salmonella enterica}; used with permission from \cite{sal}.}
\label{f:circos}
\end{figure}

Section~\ref{s:quad} begins with an overview of the space of quadratic forms, whereas Section~\ref{s:braid} recalls the relationship between configuration spaces of points on lines, the braid hyperplane arrangement, and the real moduli space of curves $\OM{n+1}$.  The heart of the paper is found in Sections~\ref{s:orb} and~\ref{s:eigen}, constructing a functor of topological groupoids relating these moduli spaces to normalized spaces of forms.   In particular, the homotopy groups are discussed, with relations to the braid and cactus groups.  Section~\ref{s:ps} closes with some biological observations.

The reader should be aware, however, that this document is a rough blueprint, which does little
more than present the definition of our blowup of the space of quadratic forms, and sketch some
of its properties. In the words of Michael Barratt, we are concerned here to create the rockpile from which the diamonds are later to be extracted. 

\begin{ack}
We reiterate our thanks detailed in Part I to the many people  and institutions who have supported
us, and this research. JM also wishes to thank Charles Epstein \cite{eps} for the suggestion that 
biological speciation is analogous to a kind of resonance. 
\end{ack}

%
%
\section{Spaces of Quadratic Forms}  \label{s:quad}

Let $\cQ_n$ be the space of real symmetric $n \times n$ matrices, equivalently, of self-adjoint operators, or quadratic forms in $n$ variables.  Among other things, it parametrizes systems of coupled harmonic oscillators, and is thus of fundamental importance in mathematical mechanics. Forty years ago, V.I.\ Arnol'd observed \cite{arn} that this (contractible) space has a very interesting stratification defined by eigenvalue multiplicities, and he remarked its relevance in applications to phenomena involving resonance. 

The geometry of spaces of matrices with conditions on their eigenvalues is complicated. Arnol'd's stratification is essentially that 
defined by the action of the orthogonal group $\Ot(n)$ of isometries on $\cQ_n$ by conjugation; quadratic forms are classified up to 
isomorphism by elements of the resulting quotient. Since
$$\dim \cQ_n - \dim \Ot(n) \ = \ \binom{n+1}{2} - \binom{n}{2} \ = \ n \;,$$
it is natural to think of an equivalence class as indexed by its \emph{unordered} configuration
\[
\sum n_i \, \{x_i\} \ \in \ \SP^n(\R) 
\]
of eigenvalues $\{x_i\}$ (counted with multiplicity) in the symmetric product $\SP^n(\R) : = \R^n/\Sg_n$. A form 
in such a class has isotropy group 
\[
\Ot(n) \supset \Ot(\bn) \; := \ \prod \; \Ot(n_i) 
\]
indexed by the partition $\bn := n_1 + \dots + n_r = \sum i \nu_i$ of $n$, where $\nu_i$ is the number of parts
with $i$ elements. The forms whose eigenvalue configurations have $r$ parts define an $r$-dimensional family, so 
the stratum indexed by $\bn$ has codimension
\[
a(\bn) \;=\; \half n(n+1) - \big[r + \half (n(n-1) - \sum n_i(n_i-1))\big] \; = \; \half \sum (i+2)(i-1)\nu_i 
\]
\cite[\S 1]{arn}. For example, there is a one-dimensional family of forms with all eigenvalues equal: it is the
subspace $\R \cdot \B$ of multiples of the identity matrix.

This function $a(\bn)$ being quadratic implies that the orbit stratification of the topological groupoid
\[
\big[\cQ_n/\Ot(n)\big]
\]
is rather singular. 
Regarding the eigenvalue configuration of a form as an element of a configuration space (compactified as in 
Part I) suggests the possible existence of a kindler, gentler stratification.

%
%
\section{The Braid Arrangement} \label{s:braid}

The symmetric group  $\Sg_n$ on $n$ letters is a finite reflection group acting on $\R^{n}$, where transpositions $(ij)$ act as reflections across the hyperplanes $\{x_i = x_j\}$.  These hyperplanes form the \emph{braid arrangement}.
Since the subspace spanned by $\la 1, \cdots, 1 \ra$ is fixed under the group action, the \emph{essential} subspace is the 
reduced regular representation of $\Sg_n$: the hyperplane $V^{n-1}$ given by $\sum x_i = 0$.
The braid arrangement decomposes the $(n-2)$-sphere $\SV{n-2}$ in $V^{n-1}$ into $n!$ simplicial chambers.   Figure~\ref{f:svpvm05}(a) shows the example of the 2-sphere $\SV{2}$ tiled by 24 simplices cut up by the braid arrangement.  Here, the symmetric group $\Sg_4$ is a Coxeter group where every conjugate of a generator $s_i$ acts on the sphere as a reflection in some hyperplane.  The tiled sphere $\SV{n}$ is then the \emph{Coxeter complex}.  In combinatorial terms, it is the barycentric subdivision of the boundary of the $(n-1)$-simplex. 

\begin{figure} [h]
\includegraphics {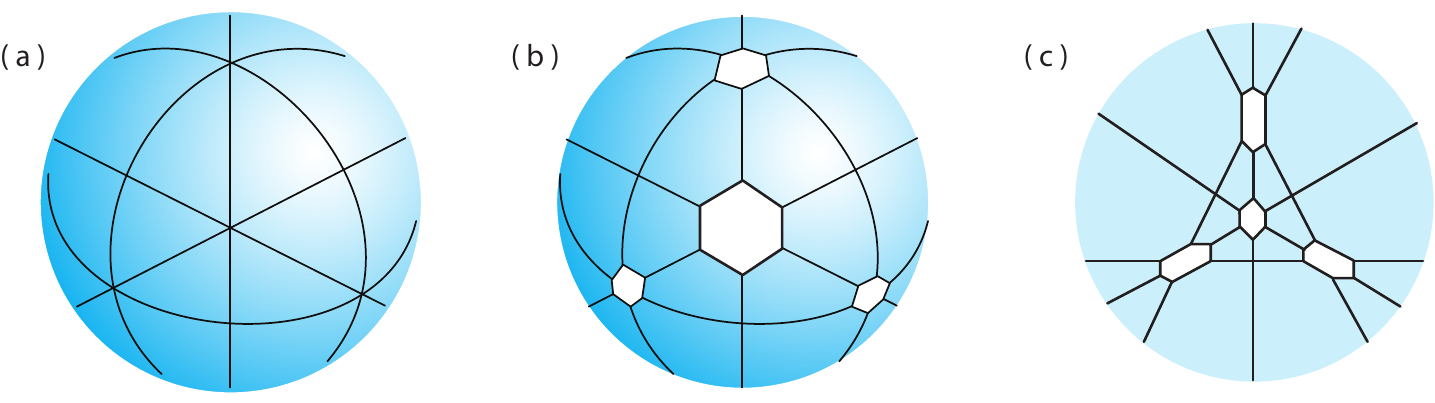}
\caption{(a) The Coxeter complex $\SV{2}$, along with (b) blowups resulting in Karpanov's nonorientable cover (c) of the real moduli space \M{5}.}
\label{f:svpvm05}
\end{figure}

Let $\Conf^n(\R)$ be the space of configurations of $n$ distinct labelled points in the line.  The contractible group 
$$\aff \; := \: \{ x \mapsto ax + b \suchthat a > 0, \ b \in \R \}$$ of orientation-preserving affine transformations of $\R$ 
acts freely on $\Conf^n(\R)$.  Translating the leftmost point to $0$, and rescaling so the rightmost point is $1$, defines an
identification of the quotient with the space of $(n-2)$ distinct labelled points in $(0,1)$.  If $(\tx_i)$ denotes the rearrangement 
of the (distinct) components of the vector $\x := (x_1, \ldots , x_n) \in \Conf^n(\R)$ in increasing order, then 
\[
(x_i) \ \mapsto \ \Big(t_k = \frac{\tx_{k+1} - \tx_k}{\tx_n - \tx_1}\Big)
\]
(so $t_k > 0$ and $\sum t_k = 1$) maps the equivalence class of $\x$ to a point of the interior of an $(n-2)$-simplex $\Delta_{n-2}$.
Thus each open simplicial chamber of the Coxeter complex \SV{n-2} corresponds to a permutation of labels for the $n$ points on the line. 
Allowing particles in $\Conf^n(\R) / \aff$ to collide completes the space to $\SV{n}$;  Figure~\ref{f:s2k4}(a) depicts a simplicial 
chamber of $\SV{2}$, along with a labeling of vertices and edges.

\begin{thm} \cite[\S 4]{dev2}
\label{t:config}
The Coxeter complex $\SV{n-2}$ can be identified with the na\"{i}ve compactification of $\Conf^{n}(\R) / \aff$.
\end{thm}

\begin{figure} [h]
\includegraphics{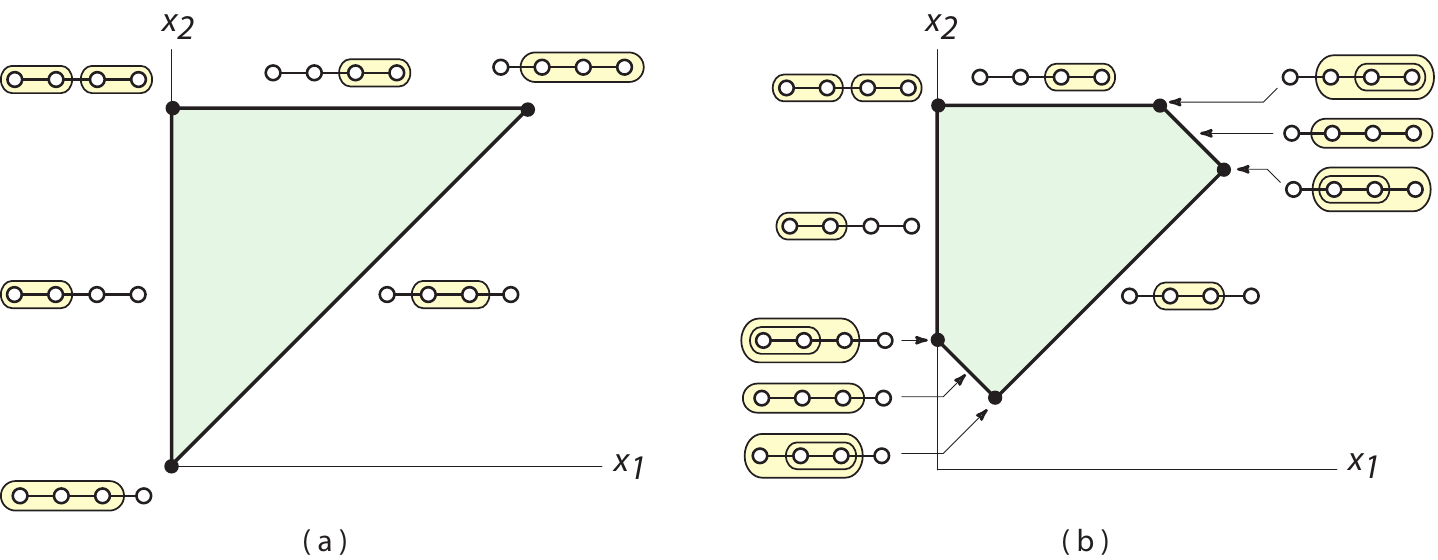}
\caption{(a) Labeling of a chamber of $\SV{2}$ and (b) the associahedron $K_4$.}
\label{f:s2k4}
\end{figure}

Kapranov \cite{kap} shows a beautiful relationship between the real moduli space of curves and the Coxeter complex of type $A$. 
Regarding the last of the $(n+1)$ points on the projective line as the point at infinity, we have:

\begin{thm} \cite[\S 4]{dev2} 
\label{t:orient}
The  moduli space \oOM{n+1} is isomorphic to $\Conf^{n}(\R) / \aff$.
\end{thm}

\noindent
Kapranov goes on to show that certain blowups of the sphere $\SV{n-2}$ result in a nonorientable double cover of \M{n+1}. The blowups
truncate each simplex tiling $\SV{n-2}$ to an associahedron $K_n$.  Gluing these associahedra under the cotwist  
operation \cite[\S 5]{dm} results in the orientable double cover \OM{n+1}, which can be viewed as a compactification of $\Conf^{n}(\R) / \aff$.

\begin{exmp}
 Figure~\ref{f:svpvm05}(b) shows blowups of $\SV{2}$, and a quotient by the antipodal map (c) resulting in \M{5}.  Compare this with the moduli space constructed from a torus in \cite[Figure 9(a)]{dm}; both are homeomorphic spaces with an identical tiling by 12 associahedra $K_4$.  Figure~\ref{f:s2k4}(b) displays the case of $K_4$ obtained from blowups of part (a).  
\end{exmp}

%
%
\section{Orbifolds and Groupoids} \label{s:orb}
\subsection{}

The group $\aff$ of affine transformations of $\R$ acts freely and compatibly with the action of the orthogonal 
group by conjugation on $\cQ_n$. The affine orbit $\R \cdot \B$ of 0 is thus invariant, and it will be convenient 
to work with the space
\[
\cQ_n^* \; := \; (\cQ_n - \R \cdot \B)/\aff
\]
of equivalence classes of quadratic forms which are nondegenerate, in the sense of having at least two distinct 
eigenvalues. Without loss of generality, we may assume that these distinct eigenvalues are 0 and 1, and that any
remaining eigenvalues lie between them. It follows from Theorem~\ref{t:config} that
\begin{cor}
We can identify $\cQ_n^*$ with a sphere of dimension $\half n(n+1) - 2$.
\end{cor}

A point $\x := (x_1, \ldots , x_n) \in \Conf^n(\R)$ defines a diagonal matrix $X^i_k = x_i \delta^i_k$,
\[
X \;=\;
\left(
 \begin{array}{ccccc}
    x_1  & & & & \text{\huge0}\\
    & \cdot  \\
    &  & & \cdot \\
    \text{\huge0} & & & & x_n
 \end{array}
\right)
\]
and hence an embedding 
\begin{equation}
\Conf^n(\R) \ \to \ \cQ_n \ : \ \x \mapsto X
\label{e:keymap}
\end{equation}
in the space of quadratic forms on $\R^n$.

\begin{thm}
The construction above defines a smooth embedding 
\begin{equation}
\oOM{n+1} \ \simeq \ \Sg_n \times \iDelta_{n-2} \ \to \ \cQ^*_n \;,
\label{e:mtoq}
\end{equation}
equivariant with respect to the inclusion $\Sg_n \to \Ot(n)$ of the symmetric group
as permutations of the coordinate axes. 
\end{thm}

\begin{proof}
Since the embedding in Eq.~\eqref{e:keymap} is equivariant under the obvious action on either side of the affine group, 
the assertion follows from Theorem~\ref{t:orient}.
\end{proof}

\subsection{}

A smooth action of a Lie group $G$ on a smooth manifold $X$ defines 
a topological category (or transformation groupoid) $[X/G]$, with $X$ as space of objects
and $G \times X$ as space of morphisms. In particular, if $x_0,x_1 \in X$, then 
\[
{\rm Maps}_{[X/G]}(x_0,x_1) \ = \ \{g \in G \:|\: gx_0 = x_1 \} 
\]
is the space of maps from $x_0$ to $x_1$. If, for example, $G$ is finite, then $[X/G]$ is an 
example of an orbifold. The geometrical realization $|[X/G]|$ of $[X/G]$, defined in terms 
of its totalization as a simplicial space, is the classifying space of the topological category; this maps
to the classical quotient $X/G$ by the homotopy-to-geometric quotient map
\[
|[X/G]| \ = \ X \times_G EG  \ \to \ X \times_G {\rm pt} \ = \ X/G \;.
\]
The construction above, with the action of Lie group $\Ot(n)$ on the smooth manifold $\cQ^*_n$, then defines a functor
\[
\big[\oOM{n+1}/\Sg_n\big] \ \to \ \big[\cQ^*_n/\Ot(n)\big]
\]
of topological groupoids; the induced map on quotients is the inclusion
$\iDelta_{n-2} \subset \Delta_{n-2}$. Compactifying as in \cite[\S 4]{dm} defines
a morphism
\[
\gM{n+1} \ := \ \big[\OM{n+1}/\Sg_n\big] \ \to \ \big[\cQ^*_n/\Ot(n)\big]
\]
which descends to a homeomorphism of either quotient space with the simplex $\Delta_{n-2}$ of normalized
eigenvalue configurations. The commutative diagram
\[
\xymatrix{
\Sg_n \times {\iDelta}_{n-2} \ \ar@{->}[d] \ar[r]^{\simeq \ \ \ }& \ \Conf^n(\R)/\aff \ar@{->}[d]\\
\partial \Delta_{n-1} \simeq \Sg V^{n-2} \ \ar@{>}[r] & \ \OM{n+1}/\Sg_n}
\]
defines a factorization
\[
\xymatrix{
\gM{n+1} := \big[\OM{n+1}/\Sg_n\big] \ar[r] & \big[\Sg V^{n-2}/\Ot(n)\big] \ar@{.>}[r] & \big[\cQ^*_n/\Ot(n)\big]}
\]
of the map above through the Coxeter complex $\Sg V^{n-2}$.

\subsection{}

As noted in \cite[\S 5]{dm}, the compactified moduli spaces \OM{n+1} have many formal similarities to the classical braid
spaces \cite[\S 4]{bv}: for example, both have an interesting $\Sg_n$-action. In the classical case this action is free, and 
its quotient space is a classifying space for the braid group ${\rm Br}_n$
on $n$ strands, defining an extension
\[
1 \ \to \ P_n :=  \pi_1(\Conf^n(\C)) \ \to \ {\rm Br}_n \ \to \ \Sg_n \ \to \ 1
\]
of fundamental groups (where $P_n$ is the group of $n$-strand pure braids).

Thurston \cite{th, md}
defined a notion of universal cover for orbifolds such as $\gM{n+1}$, and showed that
a suitable orbifold fundamental group (isomorphic to $\pi_1|[X/G]|$ when $G$ is finite) acts by
deck-transformations, generalizing the situation for topological spaces. In our case the resulting 
universal cover $\tgM{n+1}$ (which inherits a Gromov hyperbolic metric from the moduli space) has an 
action of the orbifold fundamental group
\[
1 \ \to \ \pi_1\big(\OM{n+1}\big) \ \to \ \pi_1^\orb \big(\gM{n+1}\big) := J^*_n \ \to \ \Sg_n \ \to \ 1 \;,
\]
closely related to the \emph{cactus group} of \cite{hk}.
For example, from our point of view $\tgM{5}$ (together with its $J_5$-action) is an orbifold analog 
of the classical upper half-plane\footnote{But at the time of writing, we
are unsure if $\tgM{n}$ is contractible or not.} with its canonical $\PGl_2(\Z)$-action.

%
%
\section{Quadratic forms with labelled eigenvalues} \label{s:eigen}
\subsection{}

The fiber product
\[
\xymatrix{
\cA \times_\cC \cB \ar@{.>}[d] \ar@{.>}[r] & \cB \ar[d]^G \\
\cA \ar[r]^F & \cC}
\]
of categories has as objects, triples 
\[
\{(A,B,\phi: F(A) \cong G(B)) \ \vert \   A \in \cA, \; B \in \cB, \; \phi \in {\rm Mor} \cC \} \;;
\]
a morphism $(a,b) : (A_0,B_0,\phi_0) \to (A_1,B_1,\phi_1)$ is a pair of morphisms in $\cA,\cB$ such
that the diagram 
\[
\xymatrix{
F(A_0) \ar[d]^{F(a)} \ar[r]^{\phi_0} & G(B_0) \ar[d]^{G(b)} \\
F(A_1) \ar[r]^{\phi_1} & G(B_1) }
\]
commutes. The geometric realization $|\cA \times_\cC \cB|$ of such a fiber product groupoid is homotopy equivalent to the
homotopy fiber product 
\[
|\cA| \times_{|\cC|} |\cB| \ = \ \{(a,b) \in |\cA| \times |\cB|, \; \Phi : I \to |\cA| \times |\cB| \ \vert \ \Phi(0) = a, \Phi(1) = b \}
\]
of the corresponding geometric realizations, yielding an Eilenberg-Moore spectral sequence with 
\[
E^{*,*}_2 \ = \ {\rm Tor}^{H^*|\cC|}_*(H^*|\cA|,H^*|\cB|) 
\]
for the cohomology of the fiber product, and a van Kampen theorem 
\[
\pi_1|\cA \times_\cC \cB| \ \simeq \ \pi_1|\cA| \, *_{\pi_1|\cC|} \, \pi_1|\cB|
\]
for fundamental groups. 

\subsection{} 

The pullback groupoid
\[
\xymatrix{
\sQ^*_n \ar@{.>}[d] \ar@{.>}[r] & [\cQ^*_n/1] \ar[d] \\
{\gM{n+1} = \big[\OM{n+1}/\Sg_n\big]} \ar[r] & \big[\cQ^*_n/\Ot(n)\big]}
\]
thus has diagrams 
\[
\xymatrix{
Q \ar[r]^S & X}
\]
as objects.  Here, $Q \in \cQ^*_n$ is a (normalized) quadratic form, $X$ is the diagonal matrix
indexed by a configuration $\x$ of $\OM{n+1}$ with its $(n+1)$-st point at $\infty$, 
and $S$ is an orthogonal matrix such that $SQS^{-1} = X$.  This construction results in the following:

\begin{prop}
The objects of $\sQ^*_n$ 
lying above an associahedral cell $K_n \subset \OM{n+1}$ can thus be identified with a subspace of the product
\[
(\Ot(n)/\Ot(\bn)) \times \Ot(\bn) \times K_n
\]
satisfying $r$ equations over the codimension $r$ face of $K_n$, defined by configurations 
with $\bn$ of length $r$.
\end{prop}

A morphism $(Q_0,S_0,X_0) \to (Q_1,S_1,X_1)$ is an element $\sigma \in \Sg_n$ such that
the diagram
\[
\xymatrix{
Q_0 \ar[d]^= \ar[r]^{S_0} & X_0 \ar[d]^\sigma \\
Q_1 \ar[r]^{S_1} & X_1 }
\]
commutes; in other words, conjugation by $S_1S_0^{-1}$ maps $X_0$ to $X_1$. The isotropy group of 
$(Q,S,X)$ is thus isomorphic to the commutant of $X$ in $\Ot(n)$, that is, to the product 
\[
N(\bn) \; := \; \prod N(n_i) 
\]
indexed by the partition $\bn$ defined by the repeated eigenvalues of $X$ of normalizers
\[
N(n_i) \;=\; \Sg_{n_i} \wr \Z_2
\]
(presented as wreath products) of maximal tori in $\Ot(n_i)$. 
The group $\Ot(n)$ acts on $\sQ^*_n$ by groupoid automorphisms,
with $N(\bn)$ as isotropy.   Because these isotropy groups are finite, as opposed to the isotropy groups $\Ot(\bn)$ of 
$[\cQ^*_n/\Ot(n)]$, we obtain the following:

\begin{thm}
The orbihedron $\sQ^*_n$ is a topological stack of constant relative dimension $\half n(n-1)$ over its quotient space. 
\end{thm}

\noindent
Indeed, $\sQ^*_n$ is an orbihedron \cite{ah}, and not an orbifold, because its singular points are not isolated.   
The same is true of the pullback $\tsQ^*_n$ over $\tgM{n+1}$, which we propose to call \emph{the stack of
generalized harmonic oscillators}. 

\begin{thm}
By van Kampen, $\tsQ^*_n$ is simply-connected
\[
\pi_1|\tsQ^*_n| \ \simeq \ \pi_1|\gM{n+1}| \, *_{\pi_1(B\Ot(n))} \, \pi_1(\cQ^*_n)  \ \simeq \ \{1\}
\]
if $n>2$, and it admits an action of $\Ot(n) \rtimes J^*_n$ by orbifold automorphisms.
\end{thm}

\noindent 
We expect that generators of $J_n^*$ will act on its cohomology by interesting wall-crossing formulas, but that is a topic for the future.

%
%
\section{Concluding Unscientific Postscript} \label{s:ps}
                             
Evolutionary biologists study rare events over enormous time-scales. Descent diagrams such as Darwin's tree, Figure~1 in \cite{dm}, 
go back to the beginning of modern thinking in the field: they represent incidence relations among experimentally-defined 
equivalence classes (species) in some hypothetical effectively infinite-dimensional stratified space of viable organisms. 
Branching in descent diagrams can be modeled by specialization in the sense of algebraic geometry, 
defined (for example) 
by fixing some parameter. In the language of stratified spaces this corresponds to moving from the interior of some region 
to its boundary: something like a phase change (like water to ice). In such a cartoon description, a chicken is a Dirac 
limit of a tyrannosaur, in which many of its genetic parameters tend to zero. 

At this level of vagueness, there is reason to work with codimension than with probability: evolutionary events are highly 
unlikely, and in reasonable models will have effective probability zero; but in geometry any subspace of positive codimension 
has measure, and hence probability, zero. The modern theory of phase change in condensed matter physics \cite{pwa}
has developed powerful tools for the study of such transitions (viewed as moving towards a stratum boundary, eg of some phenotype), 
but in current work there is usually only one such event in focus at a given moment. Evolution forces us to consider long 
concatenations of such events, and trees are a natural tool for their book-keeping. A geometric object with many strata 
(for example, a high-dimensional polyhedron, with faces of many dimensions) has an associated incidence graph, with a 
vertex for each face, and a directed edge between adjacent faces of lower dimension.\begin{footnote} {Computer programs 
such as Mapper \cite{map} provide something similar for data clouds, extracting simple descriptions of high dimensional data sets in the form of simplicial complexes.}\end{footnote}  From this point of view, trees 
(and more generally graphs) can provide a kind of skeletal accounting of the relations between the components of 
geometric objects which are not as simply related as manifolds are to their boundaries \cite{jm2}.  

Mathematicians are aware that objects of universal significance (such as symmetric groups) manifest 
themselves in unexpected contexts, and that their relevance to a subject can be signaled by the appearance of
related simpler objects (such as partitions, or Young diagrams). The immense utility of trees as a device
for organizing evolutionary data points in this way toward configuration spaces in genomics; but questions
there are so little understood that even very coarse models, such as those based on linear methods but with a
few extra bells and whistles coming from astute compactifications, may provide useful insights. 

%
%
\bibliographystyle{amsplain}

\end{document}